\documentclass[leqno]{amsart}
\usepackage{amsmath}
\usepackage{amssymb}
\usepackage{amsthm}
\usepackage{enumerate}
\usepackage[mathscr]{eucal}
\usepackage{mathrsfs}
\usepackage{xcolor}
\theoremstyle{plain}
\usepackage{tikz}
\newtheorem{theorem}{Theorem}[section]
\newtheorem{lemma}[theorem]{Lemma}

\theoremstyle{definition}

\newtheorem{remark}[theorem]{Remark}

\newtheorem{example}[theorem]{Example}

\newtheorem{cor}[theorem]{Corollary}
\theoremstyle{remark}

\begin{document}
	
	
	\title [Nu-Extreme contractions of $L(H)$ ]{Numerical radius norm and extreme contractions of $L(H)$ }
	
	
	\author[Arpita Mal]{Arpita Mal}
	
	\address[]{Department of Mathematics\\ Indian Institute of Science\\ Bangalore 560012\\ India.}
	\email{arpitamalju@gmail.com}

	\thanks{The author would
		like to thank SERB, Govt. of India for the financial support in the form of National Post Doctoral Fellowship
		under the mentorship of Prof. Apoorva Khare.}
	
	
	

	\date{}
	\maketitle
	\begin{abstract}
		Suppose $L(H)$ is the space of all bounded linear operators on a complex Hilbert space $H.$  This article deals with the problem of characterizing  the extreme contractions of $L(H)$ with respect to the numerical radius norm on $L(H).$ In contrast to the usual operator norm, it is proved that there exists a class of unitary operators on $H$ which are not extreme contractions when the numerical radius norm is considered on $L(H).$ Moreover, there are non-unitary operators on $H$ which are extreme contractions as far as the numerical radius norm is concerned.
	\end{abstract}

	\section{Introduction}
	The primary objective of this article is to explore the extreme contractions on a complex Hilbert space considering the numerical radius norm on the operator space. The difference of the extremal structure of an operator space endowed with the numerical radius norm and the usual operator norm is explored. 
	\\
	Suppose $H$ is a complex Hilbert space. Let $S_{H}=\{x\in H:\|x\|=1\}$ be  the unit sphere of $H.$ Suppose $L(H)$ is the space of all bounded linear operators on $H.$ For an operator  $T\in L(H),$ the numerical radius of $T,$ denoted as $w(T),$ is defined as 
	\[w(T)=\sup\{|\langle Tx,x\rangle|:x\in S_H\}.\]
	It is well-known that $w(\cdot)$ defines a norm on $L(H).$ Moreover, $w(\cdot)$ is a weakly unitarily invariant norm on $L(H),$ i.e., if $U\in L(H)$ is a unitary operator and $T\in L(H)$ is an arbitrary operator, then $w(U^*TU)=w(T).$ For each operator $T\in L(H),$
	\[\frac{1}{2}\|T\|\leq w(T)\leq \|T\|.\] An operator $T\in L(H)$ is said to be a normaloid operator \cite{GR}, if $w(T)=\|T\|$ holds. Recall from \cite[Th. 6.5-1, p. 164]{GR} that, in a two-dimensional Hilbert space, the class of all normal operators coincides with the class of all normaloid operators. However, in general, the class of all normal operators is a proper subset of the class of all normaloid operators. 
	For more information on numerical radius, go through the classical references \cite{BD,GR}.
	The collection of all unit vectors of $H,$ where an operator $T\in L(H)$ attains its numerical radius is denoted by $M_{w(T)},$ i.e.,
	\[M_{w(T)}=\{x\in S_H:|\langle Tx,x\rangle|=w(T)\}.\]
	The real part and imaginary part of an operator $T\in L(H)$ are respectively denoted as $\Re(T)$ and $\Im(T),$ i.e., $\Re(T)=\frac{1}{2}(T+T^*)$ and $\Im(T)=\frac{1}{2i}(T-T^*).$ The real part and imaginary part of a scalar $\lambda$ are denoted as $\Re(\lambda)$ and $\Im(\lambda)$ respectively. A diagonal matrix of order $n\times n,$ whose $(i,i)$ entry is $d_i$ ($1\leq i\leq n$), is denoted by $diag(d_1,d_2,\ldots,d_n).$ The symbol $I$ denotes the identity operator. Throughout the article, when an operator is written in the form of a matrix, unless otherwise mentioned, we always assume that the matrix is written with respect to an orthonormal basis. \\
	An extreme point of the unit ball of an operator space is known as an extreme contraction. The study of extreme contraction was initiated in 1951 by Kadison in \cite{K}. It is well-known  from \cite{Ga,K} that an operator is an extreme contraction of $L(H),$ endowed with the usual operator norm, if and only if it is either an isometry or a co-isometry. Recall that an operator is said to be a co-isometry, if its adjoint operator is an isometry. Extreme contractions are still widely studied on Banach spaces with respect to the usual operator norm. Interested readers may go through \cite{G,Ki,MPD,SPM,Sh} and the references therein to see the difficulty of this problem. Since $w(\cdot)$ defines a norm on $L(H),$ it is interesting to explore the extreme contractions of the normed space $\big(L(H),w(\cdot)\big).$ For other geometric properties of an operator space endowed with the numerical radius norm, see \cite{M,MPS}. We call an extreme contraction of the space $\big(L(H),w(\cdot)\big)$ by numerical radius extreme contraction, in short, nu-extreme contraction. Thus, $T\in L(H)$ with $w(T)=1$ is a nu-extreme contraction, if whenever $T=tA+(1-t)B$ holds for some $t\in (0,1)$ and $A,B\in L(H)$ where $w(A)\leq1$ and $w(B)\leq 1,$ then $T=A=B.$ Moreover, note that, for any scalar $\lambda$ with $|\lambda|=1,$ $T$ is a nu-extreme contraction if and only if $\lambda T$ is a nu-extreme contraction. On the other hand, if $T$ is weakly unitarily invariant to $A,$ then $T$ is a nu-extreme contraction if and only if $A$ is a nu-extreme contraction.
	
	In the next section, we prove that if a normaloid operator is a nu-extreme contraction, then it is either an isometry or a co-isometry. In particular, on a finite-dimensional Hilbert space, a normaloid nu-extreme contraction must be a unitary. However, the converse is not true. More generally, we show that on a finite-dimensional Hilbert space, if a self-adjoint operator is not of the form $\pm I,$ then it cannot be a nu-extreme contraction. In Theorem \ref{th-2com}, we describe the normal operators which are nu-extreme contractions on two-dimensional Hilbert space. In Theorem \ref{th-ncom}, we generalize Theorem \ref{th-2com} and characterize normal operators which are nu-extreme contractions on finite-dimensional Hilbert space. In Theorem \ref{th-inf}, we prove that on an infinite-dimensional Hilbert space, there is no compact normal operator which is a nu-extreme contraction.   We further explore nu-extreme contractions on two-dimensional Hilbert space, which are not normal operators.

	\section{Main results}
	We begin this section with a necessary condition for a normaloid operator to be a nu-extreme contraction.
	\begin{theorem}\label{th-notuni}
		Suppose $T\in L(H)$ is such that $w(T)=\|T\|=1.$ If $T$ is a nu-extreme contraction, then $T$ is either an isometry or a  co-isometry. 	\\
		In particular, if $\dim(H)<\infty,$  $w(T)=\|T\|=1$ and $T$ is a nu-extreme contraction, then $T$ is a unitary.
	\end{theorem}
	\begin{proof}
		Suppose $T$ is neither 	an isometry nor a co-isometry. 	Then by \cite{K}, $T$ is not an extreme contraction (with respect to the operator norm). Therefore, there exist $A,B\in L(H)$ such that $T=\frac{1}{2}A+\frac{1}{2}B,$ where $\|A\|=\|B\|=1$ and $T\neq A, T\neq B.$ Now, $w(A)\leq \|A\|=1$ and $w(B)\leq \|B\|=1$ prove that $T$ is not a nu-extreme contraction, a contradiction. Thus, $T$ is either an isometry or a co-isometry. 	\\
		The last part of the theorem follows from the fact that on a finite-dimensional Hilbert space,  every isometry is a unitary.
	\end{proof}

	As a consequence of Theorem \ref{th-notuni}, we get that  a positive operator, which is not the identity operator,  cannot be a nu-extreme contraction.

	\begin{cor}\label{cor-pos}
		Suppose $\dim(H)<\infty.$ Let $T\in L(H)$ be a positive operator such that $w(T)=1.$ If $T\neq I,$ then $T$ is not a nu-extreme contraction. 
	\end{cor}
	\begin{proof}
		Since $T$ is positive, $T$ is a normal operator. Therefore, $T=U^*DU$  for a unitary operator $U$ and a diagonal operator $D.$  If $T$ is a unitary, then all its eigenvalues will be $1,$ i.e., $D=I.$ Consequently $T=I,$ which is a contradiction. Therefore, $T$ is not a unitary. Now, it follows from Theorem \ref{th-notuni}  that $T$ is not a nu-extreme contraction.
	\end{proof}
	
	The following theorem provides us a large class of operators which are not nu-extreme contractions. We will use this theorem repeatedly in our next results.
	\begin{theorem}\label{th-diag}
		Let $H_1,H_2$ be  complex Hilbert spaces. Suppose $A\in L(H_1),B\in L(H_2)$ such that either of the following is true.\\
		\rm(i) $w(A)= 1$ and $A$  is not a nu-extreme contraction. \\
		\rm(ii) $w(B)= 1$ and $B$  is not a nu-extreme contraction. \\ Then for a non-zero operator $T,$ where
		\[ T=\left[ \begin{array}{cc}
			A &  O \\
			O & B
		\end{array} \right],
		\] $\frac{1}{w(T)}T$ is not a nu-extreme contraction.
	\end{theorem}
	\begin{proof}
		First assume that (i) holds. Then there exist $A_1,A_2\in L(H_1)$ with $w(A_1)\leq 1,~w(A_2)\leq 1,$ $A\neq A_1,A\neq A_2$ such that $A=\frac{1}{2}A_1+\frac{1}{2}A_2.$   Now, consider the operators
		\[ P=\left[ \begin{array}{cc}
			A_1 &  O\\
			O & B
		\end{array} \right],
		Q=\left[ \begin{array}{cc}
			A_2 &  O\\
			O & B
		\end{array} \right].
		\]
		Then from \cite[p. 10]{B}, we get
		$$w(P)=\max \{w(A_1),w(B)\}\leq \max\{1,w(B)\}= w(T).$$ Similarly, $w(Q)\leq w(T).$ Clearly, $P\neq T\neq Q.$ 
		Now, from $T=\frac{1}{2}P+\frac{1}{2}Q,$ it follows that $\frac{1}{w(T)}T$ is not a nu-extreme contraction. Similarly, if (ii) holds, then $\frac{1}{w(T)}T$ is not a nu-extreme contraction. This completes the proof of the theorem.
	\end{proof}
   Note that, Theorem \ref{th-diag} also gives us a necessary condition for an operator to be a nu-extreme contraction. However, in due course of time, we  will show that this necessary condition is not sufficient.
	Our next goal is to prove that the converse of Theorem \ref{th-notuni} is not true. For this, we first prove the following lemma.
	\begin{lemma}\label{lem-2dimself}
		Suppose $\dim(H)=2.$ Let $T\in L(H)$ be a self-adjoint operator such that $w(T)=1.$ If $T\neq \pm I,$ then $T$ is not a nu-extreme contraction.
	\end{lemma}
	\begin{proof}
		Since $T$ is a self-adjoint operator, $T=U^*DU$ for a unitary operator $U$ and a diagonal operator $D.$  Clearly, $w(D)=1.$ Suppose that $D=diag(d_1,d_2),$ where $d_1,d_2\in \mathbb{R}.$ Since $T\neq \pm I,$ we get  $d_1\neq d_2,$  Without loss of generality, we may assume that $d_1=1.$ If
		$-1<d_2<1,$ then $D$ is not a unitary and consequently $T$ is not a unitary. Therefore, from Theorem \ref{th-notuni}, it follows that $T$ is not a nu-extreme contraction. Hence,  assume that $d_2=-1.$ Consider the operators 
		\[ A=\left[ \begin{array}{cc}
			1 & i \\
			i& -1
		\end{array} \right]
		\text{ and }
		V=\frac{1}{\sqrt{2}}\left[ \begin{array}{cc}
			i & 1 \\
			1& i
		\end{array} \right].
		\]
		Observe that $V$ is a unitary operator and 
		\[ VAV^*=\left[ \begin{array}{cc}
			0 & 2i \\
			0& 0
		\end{array} \right].
		\]
		Therefore, $w(A^*)=w(A)=w(VAV^*)=1.$ Moreover, 
		\begin{eqnarray*}
			D&=&\frac{1}{2}A+\frac{1}{2}A^*\\
			\Rightarrow 	T=U^*DU&=&\frac{1}{2}U^*AU+\frac{1}{2}U^*A^*U,
		\end{eqnarray*}
		where $w(U^*AU)=w(A)=1$ and $w(U^*A^*U)=w(A^*)=1.$ Now, $D\neq A$ and $D\neq A^*$ imply that $T=U^*DU\neq U^*AU$ and $T=U^*DU\neq U^*A^*U.$
		This proves that $T$ is not a nu-extreme contraction.
	\end{proof}
	
	We generalize Lemma \ref{lem-2dimself} in the next theorem.
	\begin{theorem}\label{th-self}
		Suppose $\dim(H)<\infty.$ Let $T\in L(H)$ be a self-adjoint operator with $w(T)=1.$ Suppose that $T\neq \pm I.$ Then $T$ is not a nu-extreme contraction.
	\end{theorem}
	\begin{proof}
	The case $\dim(H)=2$ is considered in Lemma \ref{lem-2dimself}. Here we consider only the case $\dim(H)>2.$ Since $T$ is a self-adjoint operator, $T=U^*DU,$ where $U$ is a unitary operator  and $D$ is a diagonal operator.  Clearly, $w(D)=1.$ Let
		$$ D=diag(d_1,d_2,\ldots,d_n),$$
		where $d_j\in\mathbb{R} $ and $|d_j|\leq 1$ for all $1\leq j\leq n.$ Without loss of generality, assume that $|d_1|=1.$ Since  $T\neq \pm I,$ i.e., $D\neq \pm I,$ we may consider that $d_2\neq d_1.$ Let 
		\[D_1=diag(d_1,d_2)~\text{and } D_2=diag(d_3,\ldots,d_n).\] Then 
		\[ D=\left[ \begin{array}{cc}
			D_1 &  O\\
			O & D_2
		\end{array} \right].
		\]
	Clearly, $w(D_1)=1,$ $D_1$ is a self-adjoint operator and $D_1\neq \pm I.$ Thus, from Lemma \ref{lem-2dimself} it follows that $D_1$ is not a nu-extreme contraction. Hence, by Theorem \ref{th-diag}, we get $T$ is not a nu-extreme contraction. This completes the proof. 
	\end{proof}
	Note that, Theorem \ref{th-self} provides us a class of unitary operators which are not nu-extreme contractions. Whereas, each unitary operator is an extreme contraction of $L(H),$ endowed with the operator norm. This theorem highlights the difference of the geometric properties of an operator space endowed with the numerical radius norm from operator norm.\\
	
	Our next aim is to characterize unitary operators which are nu-extreme contractions. To serve our purpose, we first prove the following lemma. 
	
	\begin{lemma}\label{lem-gen}
		Let $T\in L(H)$ be such that $w(T)=1.$ Suppose $T$ is in the convex hull of $\{A_1,A_2,\ldots,A_n\},$ where $A_i\in L(H)$ and $w(A_i)=1$ for all $1\leq i\leq n.$	If $x\in M_{w(T)},$ then $x\in M_{w(A_i)}$ and $\langle Tx,x\rangle=\langle A_ix,x\rangle$ for all $1\leq i\leq n.$
	\end{lemma}
	\begin{proof}
		Let $T=\sum_{i=1}^nt_iA_i,$ where $t_i\in[0,1]$ and $\sum_{i=1}^nt_i=1.$ Then 
		\[\langle Tx,x\rangle=\sum_{i=1}^n t_i \langle A_ix,x\rangle.\]
		Since $x\in M_{w(T)},$ and $w(T)=1,$ we get $\langle Tx,x\rangle\in S_{\mathbb{C}}.$ Note that, $B_{\mathbb{C}}$	is strictly convex and $|\langle A_ix,x\rangle|\leq 1.$ Therefore, $\langle Tx,x\rangle=\langle A_ix,x\rangle$ for all $1\leq i\leq n.$ Thus, $x\in M_{w(A_i)},$ completing the proof.
	\end{proof}
	
	In the next theorem, we characterize unitary operators which are nu-extreme contractions on two-dimensional Hilbert space.
	
	\begin{theorem}\label{th-2dimuni}
		Let $\dim(H)=2.$ Suppose $T\in L(H)$	is a unitary. Then the following are true.\\
		\rm(a) If $T=\lambda I$ for a scalar $\lambda,$ then $T$ is a nu-extreme contraction.\\
		\rm(b) If $T\neq \lambda I$ for all scalars $\lambda,$ and $T^*= \frac{\mu}{\overline{\mu}}T$ for a scalar $\mu,$  then $T$ is not a nu-extreme contraction.\\
		\rm(c) If  $T\neq \lambda I$ for all scalars $\lambda,$ and $T^*\neq \frac{\mu}{\overline{\mu}}T$ for all scalars $\mu,$   then $T$ is a nu-extreme contraction.
	\end{theorem}
	\begin{proof}
		(a)  Let $T=tT_1+(1-t)T_2,$ for some $t\in(0,1)$ and $T_1,T_2\in L(H)$ with $w(T_1)=w(T_2)=1.$ Note that $x\in M_{w(T)}$ for all $x\in S_H.$ Therefore, from Lemma \ref{lem-gen}, it follows that  $\langle Tx,x\rangle=\langle T_1x,x\rangle$ for all $x\in S_H.$ Hence $T=T_1,$ since $H$ is a complex Hilbert space. Similarly, $T=T_2.$ Therefore, $T$ is a nu-extreme contraction.\\
		
		(b) Note that, $T^*= \frac{\mu}{\overline{\mu}}T$ implies that $\frac{\mu}{|\mu|} T$ is self-adjoint. Since $ T\neq \lambda I$ for all scalars $\lambda,$ $\frac{\mu}{|\mu|} T\neq \pm I.$ Therefore, from Theorem \ref{th-self}, we get $\frac{\mu}{|\mu|} T$ is not a nu-extreme contraction, thus so is $T.$\\
		
		(c) Since $T$ is a unitary, $T=U^*DU$ for a unitary operator $U$ and a diagonal operator $D.$ 	Let $D=diag(a_1+ib_1,a_2+ib_2),$ where $a_1,a_2,b_1,b_2\in \mathbb{R}$ and $|a_1+ib_1|=|a_2+ib_2|=1.$ Suppose that $D=\frac{1}{2}A+\frac{1}{2}B$ for some $A,B\in L(H)$ such that   $w(A)=w(B)=1.$ Assume that $\{e_1,e_2\}$ is an orthonormal basis of $H$ such that $De_1=(a_1+ib_1)e_1$ and $De_2=(a_2+ib_2)e_2.$ Clearly, $e_1,e_2\in M_{w(D)}.$ Now, from Lemma \ref{lem-gen}, it follows that $e_1,e_2\in M_{w(A)}\cap M_{w(B)}$ and $\langle De_j,e_j\rangle=\langle Ae_j,e_j\rangle=\langle Be_j,e_j\rangle$ for $j=1,2.$ Thus, there exist $\alpha,\beta,\gamma,\delta\in \mathbb{R}$ such that 
		\[ A=\left[ \begin{array}{cc}
			a_1+ib_1 & \gamma+i\delta \\
			\alpha+i\beta & a_2+i b_2
		\end{array} \right],~
		B=\left[ \begin{array}{cc}
			a_1+ib_1 & -\gamma-i\delta \\
			-\alpha-i\beta & a_2+i b_2
		\end{array} \right].
		\]
		Now, from \cite[Th. 2.6]{SMBP}, it follows that for $j=1,2,$ $e_j$ is an eigenvector of $\langle \Re(A)e_j,e_j\rangle \Re(A)+\langle \Im(A)e_j,e_j\rangle \Im(A)$ corresponding to the eigenvalue $w^2(A)=1.$ Thus, 
		\begin{eqnarray*}
			\langle \Re(A)e_1,e_1\rangle \Re(A)e_1+\langle \Im(A)e_1,e_1\rangle \Im(A)e_1&=&e_1\\
			\Rightarrow \langle \Re(A)e_1,e_1\rangle \langle \Re(A)e_1,e_2\rangle+\langle \Im(A)e_1,e_1\rangle \langle \Im(A)e_1,e_2\rangle&=&0\\
			\Rightarrow a_1\{(\alpha+\gamma)+i(\beta-\delta)\}+b_1\{(\beta+\delta)-i(\alpha-\gamma)\}&=&0
		\end{eqnarray*}
		\begin{eqnarray}\label{eq-uni1}
			\Rightarrow a_1(\alpha+\gamma)+b_1(\beta+\delta)=0=a_1(\beta-\delta)-b_1(\alpha-\gamma).	
		\end{eqnarray}
		Similarly, from $\langle \Re(A)e_2,e_2\rangle \Re(A)e_2+\langle \Im(A)e_2,e_2\rangle \Im(A)e_2=e_2,$ we get
		\begin{eqnarray}\label{eq-uni2}
			a_2(\alpha+\gamma)+b_2(\beta+\delta)=0=a_2(\beta-\delta)-b_2(\alpha-\gamma).	
		\end{eqnarray}
		Solving (\ref{eq-uni1}) and (\ref{eq-uni2}), we get
		\begin{eqnarray}\label{eq-uni3}
			(\beta+\delta)(a_1b_2-a_2b_1)=0.
		\end{eqnarray}
		Now, we consider two cases separately.\\
		
		Case (I). Let $a_1a_2b_1b_2\neq 0.$ 
		Note that, $w(A)=1$ gives $a_1^2+b_1^2=1=a_2^2+b_2^2.$ Now, from (\ref{eq-uni3}), if $a_1b_2-a_2b_1=0$ holds, then it is easy to check that  $a_1=\pm a_2,b_1=\pm b_2,$ i.e., either $a_1+ib_1=a_2+ib_2$ or $a_1+ib_1=-(a_2+ib_2).$ If $a_1+ib_1=-(a_2+ib_2),$ then 
		$$T^*=U^*D^*U=\frac{a_1-ib_1}{a_1+ib_1}U^*DU=\frac{a_1-ib_1}{a_1+ib_1}T,$$
		which contradicts the hypothesis. Therefore, we must have $a_1+ib_1=a_2+ib_2.$ Thus 
		$$T=U^*DU=(a_1+ib_1)I$$ and hence $T$ is a nu-extreme contraction.\\
		If $a_1b_2-a_2b_1\neq 0,$ then from (\ref{eq-uni3}), we get $\beta+\delta=0.$ Therefore, from (\ref{eq-uni1}) and (\ref{eq-uni2}), we have the following equations.
		\[\alpha+\gamma=0,~a_1\delta-b_1\gamma=0~\text{and } a_2\delta-b_2\gamma=0.\]
		Solving the above equations, we get $\alpha=\beta=\gamma=\delta=0.$ Thus $D=A=B,$ which proves that $D$ is a nu-extreme contraction.  Therefore, $T$ is also a nu-extreme contraction.\\
		
		Case (II). Let $a_1a_2b_1b_2=0.$ First assume that $a_1=0.$ Clearly, $b_1\neq 0,$ since $e_1\in M_{w(D)}.$ Note that, if $a_2=0,$ then $b_2\neq 0$ and $D=diag(ib_1,ib_2),$ which shows that $D^*=-D.$ Thus, $T^*=-T,$ contradicting the hypothesis. Therefore, if $a_1=0,$ then $b_1\neq 0$ and $a_2\neq 0.$ Now, solving (\ref{eq-uni1}), (\ref{eq-uni2}) and (\ref{eq-uni3}), we get $\alpha=\beta=\gamma=\delta=0.$ Therefore, $A=D=B.$ Thus, $D$ is a nu-extreme contraction, which eventually proves that $T$ is a nu-extreme contraction.\\
		Now, if $b_1=0,$ then using the previous argument, we can say that $iD$ is a nu-extreme contraction. Therefore, $iT$ is a nu-extreme contraction, and so is $T.$ Similarly, either $a_2=0$ or $b_2=0$ gives that $T$ is a nu-extreme contraction, completing the proof of the theorem.
	\end{proof}

	Combining Theorem \ref{th-notuni} 	and Theorem \ref{th-2dimuni}, we get the description of normaloid  (equivalently, normal) nu-extreme contractions on two-dimensional complex Hilbert spaces, which we state in the following theorem.
	\begin{theorem}\label{th-2com}
		Suppose $\dim(H)=2.$  Let $T\in L(H)$ be such that $w(T)=\|T\|=1.$  Then $T$ is a nu-extreme contraction if only if either of the following holds.\\
		\rm(i) $T=\lambda I$ for a scalar $\lambda.$\\
		\rm(ii) $T$ is a unitary and $T^*\neq \frac{\mu}{\overline{\mu}}T$ for all scalars $\mu.$
	\end{theorem}
	Now, we generalize Theorem \ref{th-2com}. We characterize those normal operators which are nu-extreme contractions on finite-dimensional Hilbert space.
	\begin{theorem}\label{th-ncom}
		Suppose $\dim(H)=n.$ Let $T\in L(H)$ be a normal operator with $w(T)=1.$ Suppose $d_1,d_2,\ldots,d_n$ are all the eigenvalues of $T.$ Then $T$ is a nu-extreme contraction if and only if for each $1\leq j<k\leq n,$ the diagonal operator $diag(d_j,d_k)$ is a nu-extreme contraction.\\
		More generally, the necessary part holds if $T$ is a normaloid operator.
	\end{theorem}
	\begin{proof}
		First suppose that $T$ is a nu-extreme contraction. Choose $1\leq j<k\leq n.$ Suppose $A=diag(d_j,d_k).$ From Theorem \ref{th-notuni}, it follows that $T$ is a unitary. Therefore, there exists a unitary operator $U$ such that 
		\[ U^*TU=\left[ \begin{array}{cc}
			A& O \\
			O & B
		\end{array} \right], 
		\]
		where $B=diag(d_1,\ldots,d_{j-1},d_{j+1},\ldots,d_{k-1},d_{k+1},\ldots,d_n),$ i.e., $B$ is a diagonal operator whose diagonal entries are the eigenvalues of $T$ except $d_j,d_k.$ Clearly, $w(A)=1.$
		Since $T$ is a nu-extreme contraction, $U^*TU$ is also a nu-extreme contraction. Therefore, from Theorem \ref{th-diag} it follows that $A$ is a nu-extreme contraction.\\
		
		Conversely, suppose that $diag(d_j,d_k)$ is a nu-extreme contraction for all $1\leq j<k\leq n.$ Since $T$ is a normal operator, there exists a unitary operator $V$ such that $V^*TV=diag(d_1,\ldots,d_n).$ Let $T=tT_1+(1-t)T_2,$ for some $t\in(0,1)$ and $T_1,T_2\in L(H),$ where $w(T_1)\leq 1$ and $w(T_2)\leq 1.$ Then
		$$diag(d_1,\ldots,d_n)=V^*TV=tV^*T_1V+(1-t)V^*T_2V.$$ Therefore, for all $1\leq j<k\leq n,$ $diag(d_j,d_k)=tP_{jk}+(1-t)Q_{jk},$ where 
		\[ P_{jk}=\left[ \begin{array}{cc}
			(V^*T_1V)_{jj}& (V^*T_1V)_{jk} \\
			(V^*T_1V)_{kj} & (V^*T_1V)_{kk}
		\end{array} \right], ~
		Q_{jk}=\left[ \begin{array}{cc}
			(V^*T_2V)_{jj}& (V^*T_2V)_{jk} \\
			(V^*T_2V)_{kj} & (V^*T_2V)_{kk}
		\end{array} \right].
		\]
		Now, from \cite[Th. 3.3]{MPS}, it follows that $$w(P_{jk})\leq w(V^*T_1V)=w(T_1)\leq 1~\text{and }w(Q_{jk})\leq w(V^*T_2V)=w(T_2)\leq 1.$$ Since $diag(d_j,d_k)$ is a nu-extreme contraction, we must have $P_{jk}=Q_{jk}=diag(d_j,d_k).$ Thus, $V^*T_1V=V^*T_2V=V^*TV,$ i.e., $T_1=T_2=T.$ This proves that $T$ is a nu-extreme contraction.\\
		Note that, in the necessary part, more generally, if we assume $T$ is a normaloid operator, then the fact that $T$ is a nu-extreme contraction and Theorem \ref{th-notuni} ensure that $T$ is a unitary. Thus, in this case, $T$ becomes a normal operator. Therefore, the necessary part also holds for normaloid operator.
	\end{proof}
 Now, we are ready to exhibit an example to show that the necessary condition of Theorem \ref{th-diag} for an operator to be a nu-extreme contraction is not sufficient.
 \begin{example}
 	Suppose $I$ is the identity operator on $\mathbb{C}^n$ and \[ T=\left[ \begin{array}{cc}
 		I&O \\
 	O&-I 
 	\end{array} \right].\] By Theorem \ref{th-ncom}, $I,-I$ are nu-extreme contractions. However, from Lemma \ref{lem-2dimself}, it follows that $diag(1,-1)$ is not a nu-extreme contraction. Therefore, by Theorem \ref{th-ncom}, $T$ is not a nu-extreme contraction, since $1,-1$ are eigenvalues of $T.$  
 \end{example} 
	In the following theorem, we prove that a compact normal operator cannot be a nu-extreme contraction, if the corresponding space is infinite-dimensional.  
	\begin{theorem}\label{th-inf}
		Suppose $H$ is an infinite-dimensional complex Hilbert space. Assume that $T\in L(H)$ is a compact normal operator and  $w(T)=1.$ Then $T$ is not a nu-extreme contraction.	
	\end{theorem}
	\begin{proof}
		First we use the spectral theorem \cite[Th. 7.6]{Con} for compact normal operators. Suppose $\{\lambda_1,\lambda_2,\ldots\}$ is the set of all distinct non-zero eigenvalues of $T$ and $P_n$ is the projection of $H$ on $\ker(T-\lambda_n I).$ Then by \cite[Th. 7.6, p. 55]{Con}, $P_mP_n=P_nP_m=0$ if $m\neq n$ and 
		\[T=\sum_{n=1}^\infty \lambda_n P_n.\] 
		Now, using \cite[Cor. 7.8 (c), p. 56]{Con}, we get $\|T\|=\sup\{|\lambda_n|:n\geq 1\},$ where either $\lambda_n\to 0$ or  $\{\lambda_n\}$ is finite. Clearly,  $|\lambda_n|=\|T\|=w(T)=1$ for some $n\in \mathbb{N}.$ Without loss of generality, we may assume that $|\lambda_1|=1.$ \\
		We now claim that either $|\lambda_m|<1$ for some $m$ or $0$ is an eigenvalue of $T.$ If possible, let $|\lambda_m|=1$ for all $m>1.$   Note that, $\ker(T-\lambda_n I)$ is finite-dimensional for each $n\in \mathbb{N},$ since $T$ is compact. Consider an orthonormal basis $B_n$ of $\ker(T-\lambda_n I).$ 
		 Since $\cup_nB_n$ is countable, $\cup_nB_n$  cannot be a Hamel basis of $H.$ We extend $\cup_nB_n$ to an orthonormal Hamel basis, say $B$ of $H.$ Choose  an element $e\in B\setminus \cup_n B_n.$ Then for all $n\in \mathbb{N},$ $e\in \ker(T-\lambda_n I)^\perp$  and thus
		\[Te=\sum_{n=1}^\infty \lambda_n P_n e=0.\]
		Hence, if $|\lambda_m|=1$ for all $m,$ then $0$ must be an eigenvalue of $T.$ This establishes the claim.\\
		Now, we can consider that $T$ has an eigenvalue $d,$ where $|d|<1.$ Suppose that $e_1\in S_H\cap \ker(T-\lambda_1 I)$ and consider an eigenvector $e_2\in S_H$ of $T$ corresponding to the eigenvalue $d.$ Let $X=span\{e_1,e_2\}.$ Then $T(X)\subseteq X.$ Assume $T|_X=A.$ Then $A$ is a normal operator. Moreover, $w(A)=1$ and $A$ is not a unitary. Now, from Theorem \ref{th-notuni}, it follows that $A$ is not a nu-extreme contraction. Let $T(X^\perp)=B(X)+C(X^\perp).$ Then 
		\[ T=\left[ \begin{array}{cc}
			A & B \\
			O & C
		\end{array} \right] \in L(X\oplus X^\perp).
		\]
		Since $TT^*=T^*T,$ we get $BB^*=O.$ Thus, $B=O.$  Now, by Theorem \ref{th-diag}, we get $T$ is not a nu-extreme contraction, since $A$ is not a nu-extreme contraction. This ends the proof of the theorem.
	\end{proof}
	
	So far, we have mostly explored nu-extreme contractions which are normal operators. Now, we concentrate on operators which are not normal. 
	To prove the next result, we state a theorem from \cite{PB}.
	\begin{theorem}\cite[Th. 2.3]{PB}\label{th-nurad}
		Let $\dim(H)<\infty.$ Suppose $A\in L(H).$   Consider the operator
		\[ T=\left[ \begin{array}{cc}
			\lambda_1I & A \\
			O & \lambda_2 I
		\end{array} \right], ~\text{where } \lambda_1,\lambda_2\in \mathbb{C}.
		\]
		Then the numerical radius of $T$ is given by $w(T)=\sqrt{x_0^2+y_0^2},$ where
		\begin{eqnarray*}
			&& \theta=\arg(\frac{\lambda_1-\lambda_2}{2}),\\
			&& h=\frac{1}{2} \Re(\lambda_1+\lambda_2),~ k=\frac{1}{2} \Im(\lambda_1+\lambda_2),\\
			&& H=h\cos\theta+k\sin\theta, ~K=h\sin\theta-k\cos\theta,\\
			&& a=\frac{1}{2}\sqrt{|\lambda_1-\lambda_2|^2+\|A\|^2}, ~b=\frac{1}{2}\|A\|,\\
			&& \frac{1}{2}(a^2-b^2)\sin(2\phi_0)=bK\cos\phi_0-aH\sin\phi_0,\\
			&& x_0=h+a\cos\phi_0\cos\theta+b\sin\phi_0\sin\theta,\\
			&& y_0=k+a\cos\phi_0\sin\theta-b\sin\phi_0\cos\theta.
		\end{eqnarray*}
	\end{theorem}
	The following theorem gives us another large class of operators which are not nu-extreme contractions.
	\begin{theorem}
		Let $\dim(H)<\infty.$ Suppose $A\in L(H)$  is not an isometry. Let $\|A\|=1.$ Consider the operator
		\[ T=\left[ \begin{array}{cc}
			\lambda_1I & A \\
			O & \lambda_2 I
		\end{array} \right] \in L(H\oplus H),~\text{where } \lambda_1,\lambda_2\in \mathbb{C}.
		\]
		Then $\frac{1}{w(T)}T$ is not a nu-extreme contraction.
	\end{theorem}
	\begin{proof}
		From  \cite{K}, it follows that $A$ is not an extreme contraction (with respect to the operator norm). Therefore, there exist operators $B,C\in L(H)$ such that $\|B\|=\|C\|=1,$ $B\neq A\neq C$ and $A=\frac{1}{2}B+\frac{1}{2}C.$  Consider  
		\[ P=\left[ \begin{array}{cc}
			\lambda_1I & B \\
			O & \lambda_2 I
		\end{array} \right],
		Q=\left[ \begin{array}{cc}
			\lambda_1I & C \\
			O & \lambda_2 I
		\end{array} \right].
		\]
		Clearly, $T=\frac{1}{2}P+\frac{1}{2}Q,$ where $P\neq T\neq Q.$ 
		From Theorem \ref{th-nurad}, it is obvious that $w(T)=w(P)=w(Q).$ Therefore, $\frac{1}{w(T)}T$ is not a nu-extreme contraction. This completes the proof.
	\end{proof}
	
	In Theorem \ref{th-2com}, we have characterized nu-extreme contractions which are normaloid operators on two-dimensional Hilbert space. In the next two theorems, we characterize nu-extreme contractions which are not normaloid operators. In the following theorem, we cover the case of those not normaloid operators $T$ for which $M_{w(T)}$ contains an orthonormal basis. 
	\begin{theorem}\label{th-2-2}
		Let $\dim(H)=2.$ Suppose $T\in L(H)$ is an operator such that $1=w(T)<\|T\|.$ Assume that $M_{w(T)}$ contains an orthonormal basis, say $\{x,y\}.$ Then $T$ is a nu-extreme contraction if and only if $|\langle Tx,y\rangle |=1.$	
	\end{theorem}
	\begin{proof}
	Recall that, for any scalar $\lambda$ with $|\lambda|=1,$ $T$ is a nu-extreme contraction if and only if $\lambda T$ is a nu-extreme contraction. Therefore, without loss of generality, we may assume that $\langle Tx,x\rangle =1.$ (For if $\langle Tx,x\rangle =\lambda,$ then we consider $\overline{\lambda}T$ instead of $T.$) We claim that $\langle Ty,y\rangle \in \mathbb{R}.$ Let $\langle Ty,y\rangle =a+ib$ for some real scalars $a,b.$ Then $a^2+b^2=1.$ Since $x,y\in M_{w(T)},$ from \cite[Th. 2.6]{SMBP}, it follows that for each $z\in \{x,y\},$  $z$ is an eigenvector of  $\langle \Re(T)z,z\rangle \Re(T)+\langle \Im(T)z,z\rangle \Im(T)$ corresponding to the eigenvalue $w^2(T)=1.$ 	Thus,
		\begin{eqnarray*}
			\langle \Re(T)x,x\rangle \Re(T)x+\langle \Im(T)x,x\rangle \Im(T)x&=&x\\
			\Rightarrow  	\langle \Re(T)x,x\rangle \langle \Re(T)x,y\rangle+\langle \Im(T)x,x\rangle \langle \Im(T)x,y\rangle&=&0\\
			\Rightarrow \Re (\langle Tx,x \rangle) \langle \Re(T)x,y\rangle+  \Im (\langle Tx,x \rangle) \langle \Im(T)x,y\rangle &=&0\\
			\Rightarrow \langle \Re(T)x,y\rangle&=&0\\
			\Rightarrow \langle Tx,y\rangle &=& -\overline{\langle Ty,x\rangle}.
		\end{eqnarray*}
		Suppose that $\langle Ty,x\rangle =\alpha.$ 
		Now, using $\langle \Re(T)y,y\rangle \Re(T)y+\langle \Im(T)y,y\rangle \Im(T)y=y$ and proceeding similarly, we get that $i\alpha b=0.$  If $\alpha=0,$ then it is easy to check that $\|T\|=1,$ a contradiction. Therefore, we must have $b=0.$ Thus, with respect to the basis $\{x,y\}$ the matrix of $T,~\Re(T)$ and $\Im(T)$ are 
		\[ T=\left[ \begin{array}{cc}
			1 & \alpha \\
			-\overline{\alpha} & a
		\end{array} \right],~
		\Re(T)=\left[ \begin{array}{cc}
			1 & 0 \\
			0 & a
		\end{array} \right],~
		\Im(T)=\left[ \begin{array}{cc}
			0 & -i\alpha \\
			i\overline{\alpha} & 0
		\end{array} \right].
		\]
		Clearly, $\Re(T)$ is an isometry and $\Im(T)$ is a scalar multiple of an isometry. \\
		We first prove the sufficient part of the theorem.  Let $|\langle Tx,y\rangle |=1.$
		If possible, suppose that $T$ is not a nu-extreme contraction. Then there are $T_1,T_2\in L(H)$ such that $T_1\neq T\neq T_2, $ $w(T_1)=w(T_2)=1$ and $T=\frac{1}{2}T_1+\frac{1}{2}T_2.$ Hence,
		\[T=\frac{1}{2}\Big(\Re(T_1)+\Re(T_2)\Big)+\frac{i}{2}\Big(\Im(T_1)+\Im(T_2)\Big),\]
		which shows that 
		\[\Re(T)=\frac{1}{2}\Big(\Re(T_1)+\Re(T_2)\Big),~\Im(T)=\frac{1}{2}\Big(\Im(T_1)+\Im(T_2)\Big).\]
			It is well-known that $w(T_1)=\sup_{\theta\in \mathbb{R}}{\|\Re(e^{i\theta}T_1)\|}$ (see \cite[Lem. 4.5]{HKS} and \cite[Proof of Th. 2.1]{Y}). Therefore, $\|\Re(T_1)\| \leq w(T_1)=1.$ Similarly, $\|\Re(T_2)\|\leq 1.$ Since $\Re(T)$ is an isometry, from \cite{K} it follows that $\Re(T)$ is an extreme contraction. Therefore, 
		\begin{eqnarray}\label{eq-isore}
			\Re(T)=\Re(T_1)=\Re(T_2).
		\end{eqnarray}
		Similarly, from $\|\Im(T_1)\|=\|\Re(e^{\frac{3i\pi}{2}}T_1)\|\leq w(T_1)=1,~ \|\Im(T_2)\|=\|\Re(e^{\frac{3i\pi}{2}}T_2)\|\leq w(T_2)=1 $ and
		the fact that $\Im(T)$ is an isometry, it follows that  
		\begin{eqnarray}\label{eq-isoim}
			\Im(T)=\Im(T_1)=\Im(T_2).	
		\end{eqnarray}
		Now, from (\ref{eq-isore}) and (\ref{eq-isoim}), we get $T=T_1=T_2,$ which is a contradiction. Thus, $T$ is a nu-extreme contraction. This completes the proof of the sufficient part.  \\
		Conversely, suppose that $T$ is a nu-extreme contraction. Observe that 
		$$0<|\alpha|=\|\Im(T)\|=\|\Re(e^{\frac{3i\pi}{2}}T)\|\leq w(T)=1.$$
		If possible, suppose that $|\alpha|<1.$ We claim that $a=-1.$ Since if $a=1,$ and $\lambda$ is an eigenvalue of $T,$ then $|\lambda|=\sqrt{1+|\alpha|^2}>1=w(T),$ a contradiction. Therefore, $a=-1.$ Now, choose a non-zero scalar $\beta$ such that $|\alpha\pm \beta|<1.$ Consider $T_1,T_2\in L(H)$ whose matrix with respect to the basis $\{x,y\}$ are 
		\[ T_1=\left[ \begin{array}{cc}
			1 & \alpha+\beta\\
			-\overline{\alpha-\beta} & -1
		\end{array} \right],~
		T_2=\left[ \begin{array}{cc}
			1 & \alpha-\beta \\
			-\overline{\alpha+\beta} & -1
		\end{array} \right].
		\]
		Clearly, $T_1\neq T\neq T_1$ and $T=\frac{1}{2}(T_1+T_2).$ Check that $T_1,T_2$ are weakly unitarily invariant to respectively
		\[ A_1=\left[ \begin{array}{cc}
			\sqrt{1-|\alpha+\beta|^2}& 2\overline{(\alpha+\beta)}\\
			0 & - 	\sqrt{1-|\alpha+\beta|^2}
		\end{array} \right]~\text{and}
		\]
		
		\[
		A_2= \left[ \begin{array}{cc}
			\sqrt{1-|\alpha-\beta|^2}& 2\overline{(\alpha-\beta)}\\
			0 & - 	\sqrt{1-|\alpha-\beta|^2}
		\end{array} \right].
		\]
		More precisely, if we consider $z_1=(\alpha+\beta) x+(-1+\sqrt{1-|\alpha+\beta|^2})y,$ and 
		\[ U=\frac{1}{\|z_1\|}\left[ \begin{array}{cc}
			(\alpha+\beta)& 1-\sqrt{1-|\alpha+\beta|^2}\\
			-1+\sqrt{1-|\alpha+\beta|^2} & \overline{(\alpha+\beta)}
		\end{array} \right],
		\]
		where $U$ is written with respect to the orthonormal basis $\{x,y\},$ then $U$ is a unitary and $U^*T_1U=A_1.$ Similarly,  if we consider $z_2=(\alpha-\beta) x+(-1+\sqrt{1-|\alpha-\beta|^2})y,$ and 
		\[ V=\frac{1}{\|z_2\|}\left[ \begin{array}{cc}
			(\alpha-\beta)& 1-\sqrt{1-|\alpha-\beta|^2}\\
			-1+\sqrt{1-|\alpha-\beta|^2} & \overline{(\alpha-\beta)}
		\end{array} \right],
		\]
		where $V$ is written with respect to the orthonormal basis $\{x,y\},$
		then $V$ is a unitary and $V^*T_2V=A_2.$
		Now, using \cite[Th. 2.3]{PB12}, we get $w(A_1)=w(A_2)=1.$ Thus, $w(T_1)=w(T_2)=1.$ This contradicts that $T$ is a nu-extreme contraction. Therefore, we must have $|\alpha|=1.$ This completes the proof of the theorem.
	\end{proof}
	Next, we give an example of a non-unitary operator which is a nu-extreme contraction. The following example once again illustrate the difference of the extremal structure of an operator space endowed with the usual operator norm and the numerical radius norm. 
	\begin{example}
		Consider the operator	
		\[
		T= \left[ \begin{array}{cc}
			1& i\\
			i & - 	1
		\end{array} \right]\in L(\mathbb{C}^2). 
		\]
		Note that, $T$ is weakly unitarily invariant to \[
		A= \left[ \begin{array}{cc}
			0& 2i\\
			0 & 0
		\end{array} \right]~(\text{see~Lemma ~\ref{lem-2dimself}}). 
		\]
		Clearly, $w(T)=w(A)=1$ and $(1,0),(0,1)\in M_{w(T)}.$ Therefore, by Theorem \ref{th-2-2}, $T$ is a nu-extreme contraction. On the other hand, since $T$ is a non-unitary operator, $T$ is not an extreme contraction of $L(\mathbb{C}^2)$ with respect to the usual operator norm.
	\end{example}
	Now, we give an example of a non-unitary operator whose numerical radius attainment set contains an orthonormal basis, still which is not a nu-extreme contraction. 
	\begin{example}
		Consider the operator	
		\[
		T= \left[ \begin{array}{cc}
			1& \frac{1}{2}\\
			-\frac{1}{2} & - 	1
		\end{array} \right]\in L(\mathbb{C}^2). 
		\]
		Note that, $T$ is weakly unitarily invariant to \[
		A= \left[ \begin{array}{cc}
			\frac{\sqrt{3}}{2}& 1\\
			0 & -	\frac{\sqrt{3}}{2}
		\end{array} \right]. 
		\]
		Thus, by \cite[Th. 2.3]{PB12}, $w(T)=w(A)=1.$ Clearly, $(1,0),(0,1)\in M_{w(T)}.$ Therefore, by Theorem \ref{th-2-2}, $T$ is not a nu-extreme contraction.
	\end{example}
	
	The only remaining case to explore, for nu-extreme contractions on two-dimensional Hilbert space, is the class of all those not normaloid operators $T,$ for which $M_{w(T)}$ does not contain any  orthonormal basis. In the next theorem, we consider this case partially. To get our desired result, we prove the following lemmas.
	\begin{lemma}\label{lem-tou}
		Suppose $\dim(H)=2.$ Let $T\in L(H)$ and $w(T)=1.$ Suppose 
		\[T=\left[ \begin{array}{cc}
			\beta & \zeta \\
			0 & \beta
		\end{array} \right],~\text{where~} \beta,\zeta\in \mathbb{C}\setminus\{0\}.
		\]
		Then $T$ is not a nu-extreme contraction.
	\end{lemma}
	\begin{proof}
		Suppose that $\beta=e^{i\theta}|\beta|,$ for some $\theta\in[0,2\pi).$ Let $\frac{|\zeta|}{|\beta|}=k(>0)$ and $t=\frac{2}{k+2}.$ Then $t\in(0,1).$ Consider 
		\[A_1=\left[ \begin{array}{cc}
			\frac{|\beta|}{t} & 0 \\
			0 & \frac{|\beta|}{t} 
		\end{array} \right],~~
		A_2=\left[ \begin{array}{cc}
			0 & \frac{e^{-i\theta}\zeta}{1-t} \\
			0 & 0
		\end{array} \right].
		\]
		Clearly, $w(A_1)=\frac{|\beta|}{t}=\frac{|\beta|(k+2)}{2}$ and $w(A_2)=\frac{|\zeta|}{2(1-t)}=\frac{k|\beta|}{2(1-t)}=\frac{|\beta|(k+2)}{2}.$ From \cite[Th. 2.3]{PB12}, it follows that
		\[w(e^{-i\theta}T)=\frac{1}{2}\Big\{2|\beta|+|\zeta|\Big\}=\frac{1}{2}\Big\{2|\beta|+k|\beta|\Big\}=\frac{|\beta|(k+2)}{2}.\] 
		Observe that $1=w(T)=w(e^{-i\theta}T)=w(A_1)=w(A_2).$ Moreover, 
		\[e^{-i\theta}T=tA_1+(1-t)A_2, \text{~and~}A_1\neq e^{-i\theta}T\neq A_2.\]
		Therefore, $e^{-i\theta}T$ is not a nu-extreme contraction. Consequently, $T$ is not a nu-extreme contraction. This completes the proof of the lemma.
		
	\end{proof}

	We use the following theorem from \cite{JSG} in our next lemma.
	
			\begin{theorem}\label{th-Johnson}\cite[Th. 1.1]{JSG}
		Let $\dim(H)=2$ and $A\in L(H).$ Suppose $A$ is unitarily invariant to 
		\[\left[ \begin{array}{cc}
				\lambda_1 & \zeta \\
				0 & \lambda_2
			\end{array} \right].
			\]
				Suppose $|\lambda_1|=|\lambda_2|~(=|\lambda|),$ but $\lambda_1\neq \lambda_2.$ Then 
				\begin{equation*}
					w(A)=
					\begin{cases}
						|\lambda|(z\sin{\theta}+\cos{\theta}) & \text{if } z \geq \tan{\theta}\\
						|\lambda|\sqrt{1+z^2} & \text{if } 0\leq z \leq \tan{\theta}.
						\end{cases}
					\end{equation*}
			Here, $\theta$ is a half angle between the directions of  $\lambda_1$ and $\lambda_2,$ $\theta\in(0,\frac{\pi}{2}]$ and $z=\frac{|\zeta|}{2|\lambda|\sin{\theta}}.$
			\end{theorem}
		
		In the next lemma, we obtain the numerical radius of certain class of operators. 
		\begin{lemma}\label{lem-wt}
			Let $\dim(H)=2.$ Suppose 
				\[ T=\left[ \begin{array}{cc}
				1 & \alpha \\
				-\overline{\alpha} & a
			\end{array} \right]\in L(H), ~\text{where~} a\in \mathbb{R},~ |a|<1 \text{~and~}\alpha \in \mathbb{C}\setminus\{0\}.
			\]
			Suppose either of the following holds.\\
			\rm(i) $(1-a)^2<4|\alpha|^2\leq 2(1-a).$\\
			\rm(ii) $(1-a)^2>4|\alpha|^2.$\\
			Then $w(T)=1.$
		\end{lemma} 
		\begin{proof}
			(i) Let $K=\sqrt{4|\alpha|^2-(1-a)^2}.$ Suppose $T$ is written with respect to the orthonormal basis, say $\{e_1,e_2\}.$ Let $x=2\alpha e_1+(-1+a+iK)e_2.$ Consider the unitary operator  
				\[ U=\frac{1}{\|x\|}\left[ \begin{array}{cc}
				2\alpha& 1-a+iK\\
				-1+a+iK& 2 \overline{\alpha}
			\end{array} \right],
			\]
			where $U$ is written with respect to the orthonormal basis $\{e_1,e_2\}.$ 
		Then 
			\[ U^*TU=\left[ \begin{array}{cc}
			\frac{1+a+iK}{2}& \frac{(1-a)(1-a+iK)}{2\alpha}\\
			0& \frac{1+a-iK}{2}
		\end{array} \right]=\left[ \begin{array}{cc}
		\lambda& \zeta\\
		0& \overline{\lambda}
	\end{array} \right],
		\]
		where $\lambda =\frac{1+a+iK}{2}$ and $\zeta=\frac{1}{2\alpha}(1-a)(1-a+iK).$
		We use Theorem \ref{th-Johnson} to obtain $w(T).$ Let $\theta$ be a half angle between the directions of $\lambda$ and $\overline{\lambda},$ $\theta\in (0,\frac{\pi}{2}].$ Then 
		\[\tan \theta=\frac{K}{1+a}, ~\sin\theta=\frac{K}{2|\lambda|},~\cos\theta=\frac{1+a}{2|\lambda|}.\]
		Suppose $z=\frac{|\zeta|}{2|\lambda|\sin{\theta}}=\frac{1-a}{K}.$ Now, using  $2(1-a)\geq4|\alpha|^2,$ we get $z\geq\tan\theta.$ Thus, by Theorem \ref{th-Johnson}, 
		\[w(T)=w(U^*TU)=|\lambda|(z\sin\theta+\cos\theta)=\frac{|\zeta|}{2}+|\lambda|\cos\theta=\frac{1-a}{2}+\frac{1+a}{2}=1,\]
		completing the proof of (i). \\
		(ii) Let $R=\sqrt{(1-a)^2-4|\alpha|^2}.$ Suppose $T$ is written with respect to the orthonormal basis, say $\{e_1,e_2\}.$ Let $x=2\alpha e_1+(-1+a+R)e_2.$ Consider the unitary operator  
		\[ V=\frac{1}{\|x\|}\left[ \begin{array}{cc}
			2\alpha& 1-a-R\\
			-1+a+R& 2 \overline{\alpha}
		\end{array} \right],
		\]
		where $V$ is written with respect to the orthonormal basis $\{e_1,e_2\}.$ 
		Then 
		\[ V^*TV=\left[ \begin{array}{cc}
			\frac{1+a+R}{2}& 2\overline{\alpha}\\
			0& \frac{1+a-R}{2}
		\end{array} \right].
		\]
		Now, by \cite[Th. 2.3]{PB12}, we get
		\[w(T)=w(V^*TV)=\frac{1}{2}\Big\{\Big|\frac{1+a+R}{2}+\frac{1+a-R}{2}\Big|+\sqrt{R^2+4|\alpha|^2}\Big\}=1,\]
	completing the proof of (ii). 
      \end{proof}
  
  	Now, we are ready to prove our final theorem.
  
\begin{theorem}\label{th-2fin}
	Let $\dim(H)=2.$ Suppose $T\in L(H)$ is an operator such that $1=w(T)<\|T\|.$ Assume that $M_{w(T)}$ does not contain any orthonormal basis. Moreover, suppose that there exist $x\in M_{w(T)}$ such that the following are true.\\  
	\rm(i) Either $\arg{\langle Tx,x\rangle}=\pi+\arg{\langle Ty,y\rangle}$ or $\arg{\langle Tx,x\rangle}=\arg{\langle Ty,y\rangle},$ where  $x\perp y,$ $y\in S_H.$\\
	\rm(ii) $2|\langle Tx,y\rangle|^2\neq 1-\frac{\langle Ty,y\rangle}{\langle Tx,x\rangle},$ where  $x\perp y,$ $y\in S_H.$\\
	Then $T$ is not a nu-extreme contraction. 
\end{theorem}
\begin{proof}
	Without loss of generality, we may assume that $\langle Tx,x\rangle=1.$ Assume that $\langle Ty,y\rangle =a.$ Then by (i), $a\in \mathbb{R}.$ Since $y\notin M_{w(T)},$ $|a|<1.$ Using the fact that $x\in M_{w(T)},$ 	and proceeding similarly as Theorem \ref{th-2-2}, we get that $\langle Tx,y\rangle=-\overline{\langle Ty,x\rangle}.$ Suppose that $\langle Ty,x\rangle=\alpha.$ Thus, with respect to the orthonormal basis $\{x,y\}$ the matrix of $T$ is 
	\[\left[ \begin{array}{cc}
		1 & \alpha \\
		-\overline{\alpha} & a
	\end{array} \right].
	\]
	We consider the following cases separately.
	\[\text{Case~ I}: 4|\alpha|^2=(1-a)^2, ~~~\text{~Case~ II}: 4|\alpha|^2<(1-a)^2,~~~\text{~Case~III}:4|\alpha|^2>(1-a)^2.\]
	Case I: Let $4|\alpha|^2=(1-a)^2.$ In this case, proceeding similarly as Lemma \ref{lem-wt} (ii), we can show that $T$ is weakly unitarily invariant to 
	\[\left[ \begin{array}{cc}
		\frac{1+a}{2}& 2\overline{\alpha}\\
		0& \frac{1+a}{2}
	\end{array} \right].
	\]
	Thus, by Lemma \ref{lem-tou}, $T$ is not a nu-extreme contraction.\\
	
	Case II: Let $4|\alpha|^2<(1-a)^2.$ Choose  a non-zero scalar $\beta$ such that $4|\alpha\pm \beta|^2<(1-a)^2.$ Consider $T_1,T_2\in L(H),$ whose matrix with respect to the basis $\{x,y\}$ are respectively
	\[ T_1=\left[ \begin{array}{cc}
		1 & \alpha +\beta\\
		-\overline{\alpha-\beta} & a
	\end{array} \right], ~\text{and~}
	T_2=\left[ \begin{array}{cc}
		1 & \alpha -\beta\\
		-\overline{\alpha+\beta} & a
	\end{array} \right]. 
	\]
	Then $T=\frac{1}{2}(T_1+T_2)$ and  $T_1\neq T\neq T_2.$ Now, by Lemma \ref{lem-wt} (ii), we get $w(T_1)=w(T_2)=1.$ Therefore, $T$ is not a nu-extreme contraction.\\
	
	Case III: Let $4|\alpha|^2>(1-a)^2.$
	It is well-known that $w(T)=\sup_{\theta\in \mathbb{R}}{\|\Re(e^{i\theta}T)\|}.$ Note that,
	\[ \Re(e^{i\theta}T)=\left[ \begin{array}{cc}
		\cos{\theta} & i\alpha\sin{\theta} \\
		-i\overline{\alpha}\sin{\theta} & a\cos{\theta}
	\end{array} \right].
	\]
	Now, the eigenvalues of $\Re(e^{i\theta}T)$ are 
	\[\frac{(1+a)\cos{\theta}\pm \sqrt{(1-a)^2\cos^2{\theta}+4|\alpha|^2\sin^2{\theta}}}{2}.\]
	Therefore, whenever $\cos{\theta}\geq 0,$ 
	\[\|\Re(e^{i\theta}T)\|=\frac{(1+a)\cos{\theta}+ \sqrt{(1-a)^2\cos^2{\theta}+4|\alpha|^2\sin^2{\theta}}}{2}.\]
	Note that, if $2|\alpha|^2+a-1>0$ holds, then there exists $n\in \mathbb{N},$ such that 
	\begin{equation*}
		n(2|\alpha|^2+a-1)>a+|\alpha|^2
	\end{equation*}
	holds. Now, using the last inequality, after an elementary calculation, we get 
	\begin{eqnarray}\label{eq-tou1}
		\frac{(1+a)(1-\frac{1}{n})+\sqrt{(1-a)^2(1-\frac{1}{n})^2+4|\alpha|^2(\frac{2}{n}-\frac{1}{n^2})}}{2}>1.
	\end{eqnarray}
	Choose $\theta\in \mathbb{R}$ such that $\cos{\theta}=1-\frac{1}{n}.$ Then $\sin^2{\theta}=\frac{2}{n}-\frac{1}{n^2}$ and from (\ref{eq-tou1}), we get
	\[\frac{(1+a)\cos{\theta}+ \sqrt{(1-a)^2\cos^2{\theta}+4|\alpha|^2\sin^2{\theta}}}{2}>1,\]
	i.e., $\|\Re(e^{i\theta}T)\|>1=w(T),$ which is a contradiction. Therefore, we must have
	\begin{equation}\label{eq-tou02}
		2|\alpha|^2+a-1\leq0.
	\end{equation}
Now, using (ii) and (\ref{eq-tou02}), we get	$2|\alpha|^2+a-1<0.$ Choose a non-zero scalar $\beta$ such that 
\begin{equation}\label{eq-tou2}
	2|\alpha\pm \beta|^2+a-1<0 ~\text{and ~} 4|\alpha\pm \beta|^2\neq (1-a)^2.
	\end{equation}
Consider $T_1,T_2\in L(H),$ whose matrix with respect to the basis $\{x,y\}$ are respectively
\[ T_1=\left[ \begin{array}{cc}
	1 & \alpha +\beta\\
	-\overline{\alpha-\beta} & a
\end{array} \right], ~\text{and~}
T_2=\left[ \begin{array}{cc}
	1 & \alpha -\beta\\
	-\overline{\alpha+\beta} & a
\end{array} \right]. 
\]
Then $T=\frac{1}{2}(T_1+T_2)$ and  $T_1\neq T\neq T_2.$ Now, using (\ref{eq-tou2}) and  Lemma \ref{lem-wt}, we get $w(T_1)=w(T_2)=1.$ Therefore, $T$ is not a nu-extreme contraction. This completes the proof of the theorem. 
\end{proof}

	\begin{remark}
	Once again, we would like to emphasize that  Theorem \ref{th-2com} provides us the description of nu-extreme contractions which are normaloid operators. Theorem \ref{th-2-2} provides us the description of nu-extreme contractions which are not normaloid operators and the numerical radius attaining set of the operator contains an orthonormal basis.  In Theorem 	\ref{th-2fin}, we have considered the remaining case, i.e., the case of all not normaloid operators whose numerical radius attaining set does not contain any orthonormal basis. Here we additionally assume some restrictions and get that the operator is not a nu-extreme contraction.  The last theorem also indicates that the study of not normaloid nu-extreme contractions is complicated even when the underlying space is two-dimensional. 
\end{remark}

		\bibliographystyle{amsplain}

	\end{document}